\documentclass[11pt]{article}
\usepackage{tikz}
\usetikzlibrary{arrows,calc,shapes,decorations.pathreplacing}
\usepackage{pgfplots}
\usepackage{comment}
\usepackage{amsmath}
\pgfplotsset{compat=newest}
\usepackage{amssymb}
\usepackage[hidelinks]{hyperref}
\usepackage{algorithm}
\usepackage{algorithmic}

\newcounter{remark}[section]

\def\claim{\par\medskip\noindent\refstepcounter{remark}\hbox{\bf Remark \arabic{section}.\arabic{remark}}
\ 
}
\def\endclaim{
\par\medskip}
\newenvironment{remark}{\claim}{\endclaim}

\newcounter{example}[section]

\def\claimex{\par\medskip\noindent\refstepcounter{example}\hbox{\bf Example \arabic{section}.\arabic{example}}
	\ 
}
\def\endclaimex{
	\par\medskip}
\newenvironment{example}{\claimex }{\endclaimex}

\newtheorem{theorem}{Theorem}

\newtheorem{lemma}{Lemma}

\newtheorem{corollary}{Corollary}

\def\r{\rho}
\def\a{\alpha}

\def\l{\lambda}

\def\m{\mu}
\def\s{\sigma}

\def\d{\partial}
\def\endpf{{\ \hfill\hbox{\vrule width1.0ex height1.0ex}\parfillskip 0pt
}}

\newenvironment{proof}{\noindent{\bf Proof:}}{\endpf}
\begin{document}
\title{A rate balance principle and its application to queueing models}
\author{Binyamin Oz\footnote{Department of Statistics and Federmann Center for the Study of Rationality, The Hebrew University of Jerusalem}, Ivo Adan\footnote{Department of Industrial Engineering, Technische Universiteit Eindhoven} ~and Moshe Haviv\footnotemark[1]
}
\maketitle

\begin{abstract}

We introduce a rate balance principle for general (not necessarily Markovian) stochastic processes. Special attention is given to processes with birth and death like transitions, for which it is shown that for any state $i$, the rate of two consecutive transitions from $i-1$ to $i+1$, coincides with the corresponding rate from $i+1$ to $i-1$. This observation appears to be useful in deriving well-known, as well as new, results for the Mn/Gn/1 and G/Mn/1 queueing systems, such as a recursion on the conditional distributions of the residual service times (in the former model) and of the residual inter-arrival times (in the latter one), given the queue length.

\end{abstract}

%

\section{Introduction}

Consider a (not necessarily Markovian) stochastic process with $\mathcal{S}$ as its state-space and partition it in three sets:
$\mathcal{D}$, $\mathcal{M}$, and $\mathcal{U}$. We define an {\it  up path segment} as a path segment which commences in some state in $\mathcal{D}$, ends in $\mathcal{U}$, and uses as intermediate states (if any) only states in $\mathcal{M}$. In a similar fashion, we define a {\it down path segment}. We show that in any (finite) time interval, the number of up and down path segments differ by at most one. This implies that their steady state rates (if they exist) coincide. The case $\mathcal{M}=\emptyset$ is the well known balance principle between the sets $\mathcal{D}$ and $\mathcal{U}$. Another special case, on which we dwell in the sequel, are processes with birth and death like transitions.
In this case, with $\mathcal{M}$ consisting of the single state $i$, what we get is that the rate of two consecutive transitions from $i-1$ to $i+1$ (namely, transitions that avoid getting again into state $i-1$ before reaching $i+1$), equals the rate of corresponding ones from $i+1$ to $i-1$. Note, however, that this result does not extend to states further than two transitions away from each other. Through a number of examples, we show the usefulness of this special case in deriving known, and also new, results in single server queues. In particular, it leads to an alternative derivation for the limiting probabilities in M/G/1 and G/M/c queues and to less known results on the residual service and inter-arrival times given the queue lengths.

Section~\ref{RBPsec} states the main result which we call the {\it rate balance principle} (RBP). A few examples are given as well. Section~\ref{notations}  presents some preliminaries on the distribution of the residual of a random variable given it is larger than an independent exponentially distributed random variable. Section~\ref{GM1} shows how the limiting probabilities of the G/M/1 queue can be derived using the RBP. This is repeated for the G/Mn/1 queue in Section~\ref{GMn1}. We also derive a recursion on the distribution function of the residual of the inter-arrival times at departure instances given the queue length. In Section~\ref{MG1} we derive the corresponding results for the Mn/Gn/1 queue, in which case the residuals of the service times at arrival instances are of concern. 
To the best of our knowledge, the recursions on the residuals of the conditional inter-arrival times and service times are new.
Finally, Section~\ref{summ} concludes.

\section{Rate balance principle}\label{RBPsec}

Let $X=\{X(t),t\ge 0\}$ be a
continuous-time stochastic process with state
space $\mathcal{S}$. As in Section 3.2 of \cite{TahaStid}, we assume that $\mathcal{S}$ is a Polish (complete separable metric) space, with the Borel  $\s$-field  $\mathcal{B}(\mathcal{S}$),  and  that  $X$  is right continuous with left-hand limits. Let $\mathcal{D}$ and $\mathcal{U}$ be two non-empty and disjoint subsets of the state space,  $\emptyset \subsetneq \mathcal{D},\mathcal{U}\in \mathcal{B}(\mathcal{S})$, and let
$\mathcal{M}=\mathcal{S}\setminus (\mathcal{D}
\cup \mathcal{U})$.
We are interested in two types of path segments of this process. The first type are path segments that begin with a state in $\mathcal{D}$, end with a state in $\mathcal{U}$, and any other states in the segments (if any) are in $\mathcal{M}$. The second type are path segments that begin in $\mathcal{U}$, end in $\mathcal{D}$ and any other states (if any) are in $\mathcal{M}$. We refer to such path segments as $\mathcal{U}$-segments and $\mathcal{D}$-segments, respectively. More formally, let $\{I^\mathcal{U}_n,n\ge 1\}$ and $\{I^\mathcal{D}_n,n\ge 1\}$ be two point processes indicating the time instances where $X$ gets into $\mathcal{U}$ and $\mathcal{D}$, respectively, i.e.,
$$I^\mathcal{V}_n=\inf\{t>I^\mathcal{V}_{n-1}|X(t^-)\notin \mathcal{V},X(t)\in \mathcal{V}\},~~\mathcal{V}\in\{\mathcal{U},\mathcal{D}\},~~n\ge 1 ,$$
where $I^\mathcal{V}_0:=0$, $\mathcal{V}\in\{\mathcal{U},\mathcal{D}\}$.
Assume that the partition of $\mathcal{S}$ is such that a.s. $X$ gets into $\mathcal{U}$ and $\mathcal{D}$ infinitely often in $[0,\infty)$, but at most finitely often in every finite time interval $[0,t)$, $t\ge 0$.

A $\mathcal{U}$-segment end point is the first time instant that $X$ gets into $\mathcal{U}$, after getting into $\mathcal{D}$, and a $\mathcal{D}$-segment end point is defined symmetrically. Formally, the counting process of $\mathcal{U}$-segments and $\mathcal{D}$-segments are defined by,
$$
N^\mathcal{U}(t)=\#\{k|I^\mathcal{U}_k\le t,\exists m \text{ s.t. } I^\mathcal{U}_{k-1}<I^\mathcal{D}_{m}<I^\mathcal{U}_{k}\}
$$
and
$$
N^\mathcal{D}(t)=\#\{k|I^\mathcal{D}_k\le t,\exists m \text{ s.t. } I^\mathcal{D}_{k-1}<I^\mathcal{U}_{m}<I^\mathcal{D}_{k}\},
$$
respectively. We denote by $\{T^\mathcal{V}_n,n\ge 1\}$, the point process associated with $N^\mathcal{V}$, $\mathcal{V}\in \{\mathcal{U},\mathcal{D}\}$.

The following theorem states that the steady state rates at which the two types of path segments occur are equal.
\begin{theorem}\label{RBP}
The following limits, if they exist, are equal:
 $$\lim_{t\to \infty} \frac{N^{\mathcal{U}}(t)}{t} = \lim_{t\to \infty} \frac{N^{\mathcal{D}}(t)}{t}.$$
\end{theorem}
\begin{proof}
	Suppose that $I^\mathcal{U}_k$ is a $\mathcal{U}$-segment end point. We argue that the previous segment end point is a $\mathcal{D}$-segment end point. Let $m^*=\min\{m|I^\mathcal{U}_{k-1}<I^\mathcal{D}_{m}<I^\mathcal{U}_{k}\}$. Note that $m^*$ exists since it is the minimum of a finite, non-empty set. Observe that $I^\mathcal{D}_{m^*-1}<I^\mathcal{U}_{k-1}<I^\mathcal{D}_{m^*}$ and hence, $I^\mathcal{D}_{m^*}$ is a $\mathcal{D}$-segment end point. The fact that the previous $\mathcal{U}$-segment end point is less than or equal to $I^\mathcal{U}_{k-1}$ completes the argument. Reversing  the  roles of $\mathcal{U}$ and $\mathcal{D}$ gives a symmetric argument.

The above implies that the two types of path segments occur alternatingly and hence,
$N^\mathcal{U}(t)$ and $N^\mathcal{D}(t)$ differ at most by $1$ for all $t \ge 0$ and thus
	$$
	\lim_{t\to \infty} \frac{N^\mathcal{U}(t)-N^\mathcal{D}(t)}{t} = 0,
	$$
	which completes the proof.
\end{proof}
\\

Practically speaking, in order to use RBP, one should consider sets $\mathcal{D}$ and $\mathcal{U}$ such that $\mathcal{U}$-segments and $\mathcal{D}$-segments have a simple structure. For example, consider a partition $\mathcal{D}, \mathcal{U}, \mathcal{M}$, and consider the following directed graph: the set of vertices is $\mathcal{M}$, and the set of edges is $\{(i,j):i,j\in \mathcal{M}\text{ and }i\to j\text{ is a possible transition}\}$. If this graph is acyclic, then for each $d \in \mathcal{D}$ and each $u \in \mathcal{U}$ there are finitely many path segments emanating from $d$ and ending in $u$ (and vice versa). We give a couple of examples for such choice of $\mathcal{D}$ and $\mathcal{U}$ below.
\\

\begin{example}\label{RBE} {\bf Rate-Balance Equation} \\
This example is based on Theorem 3.7 in \cite{TahaStid}. Let $A\in \mathcal{B}(\mathcal{S})$ and let $\mathcal{D}=A$ and $\mathcal{U}=A^c$ (so $\mathcal{M} = \emptyset$). Of course, Theorem \ref{RBP} implies that the rate of transitions from $A$ to $A^c$, i.e., the rate of transitions out of $A$, equals the rate of transitions from $A^c$ to $A$, i.e., the rate of transitions into the set $A$.
In case $\mathcal{S} = \mathbb{Z}$ and $A=\{k:k\in \mathbb{Z},k\le \ell\}$ for some $\ell\in \mathbb{Z}$, we get the classical \emph{level crossing} argument.
\end{example}

\begin{example}{\bf Two-Step Transitions (TST)}\\
Consider a process with state space $\mathcal{S} = \mathbb{Z}^+$ where transitions are of size 1, e.g., the number of customers in a queueing system where customers arrive one by one and are served one at a time. For $n\ge 1$, let $\mathcal{D}=\{k:0\le k <n\}$ and $\mathcal{U}=\{k:k>n \}$. Here, $\mathcal{U}$-segments have the following form: they begin with a transition from state $n-1$ to state $n$, and end with a transition from state $n$ to state $n+1$, i.e., two consecutive up transitions. We refer to such path segments as $n$-two-step up transitions. Symmetrically, $\mathcal{D}$-segments have the form of two consecutive down transitions, from state $n+1$ to state $n$, and then from state $n$ to state $n-1$. We refer to such path segments as $n$-two-step down transitions. Theorem \ref{RBP} implies that the rates of $n$-two-step up transitions and $n$-two-step down transitions are equal.
In that case, as in the general case, the reason for this equality is that the two transitions occur alternately. Suppose that an $n$-two-step up transition just occurred. That means that the process is currently in state $n+1$. In order for another $n$-two-step up transition to occur, the process must visit state $n-1$ first, and for that to happen, an $n$-two-step down transition must occur. Figure \ref{two-step} shows an example of a sample path where 3-two-step transitions are marked with a solid line.
	\begin{figure}[h!]
		\centering
		\begin{tikzpicture}[x=0.6cm,y=0.45cm]
		
		\def\xmin{0}
		\def\xmax{14}
		\def\ymin{0}
		\def\ymax{6}
					
		\draw[->] (\xmin,\ymin) -- ({\xmax*1.1},\ymin) node[below] {$t$};

		\draw[->] (\xmin,{\ymin}) -- (\xmin,{\ymax*1.1}) node[left] {$X(t)$};
	
		\foreach \i in {0,1,...,5}
		\draw[-] (\xmin,\i) -- ({\xmin-0.1},\i) node[left] {{\i}};
		\draw[-,dashed] (0,3) -- ({\xmax*1.05},3);
		\def\n{0}
		\def\u{0}
		
		\newcommand\jump[2]{
			\edef\d{#1}
			\edef\m{#2}
			\draw[-,thin] (\u,\n) -- (\u+\d,\n) -- (\u+\d,\n+\m);
			\edef\u{\u+\d}
			\edef\n{\n+\m}
		}
		
		\newcommand\emphjump{ \draw[-,very thick] (\u,\n-\m) -- (\u,\n) }
		
		\newcommand\markjump[1]{ \draw[-,dotted] (\u,\n-\m) -- (\u,{\ymin}) node[below]{#1}}
		
		\jump{0.8}{1};
		\jump{0.2}{1};
		\jump{1}{-1};
		\jump{0.6}{1};
		\jump{0.8}{1};
		\jump{0.2}{-1};
		\jump{0.9}{1};
		\emphjump;
		\jump{0.2}{1};
		\emphjump;
		\markjump{$T_1^{\mathcal{U}}$};
		\jump{0.7}{1};
		\jump{1.2}{-1};
		\jump{0.2}{-1};
		\jump{1.1}{1};
		\jump{0.4}{-1};
		\emphjump;
		\jump{0.7}{-1};
		\emphjump;
		\markjump{$T_1^{\mathcal{D}}$};
		\jump{0.9}{-1};
		\jump{0.3}{1};
		\jump{0.5}{1};
		\emphjump;
		\jump{0.3}{1};
		\emphjump;
		\markjump{$T_2^{\mathcal{U}}$};
		\jump{0.7}{1};
		\jump{0.5}{-1};
		\jump{0.4}{-1};
		\emphjump;
		\jump{1}{-1};
		\emphjump;
		\markjump{$T_2^{\mathcal{D}}$};
		\jump{0.7}{-1};	
		\end{tikzpicture}
		\caption{A sample path with 3-two-step transitions}  \label{two-step}
	\end{figure}
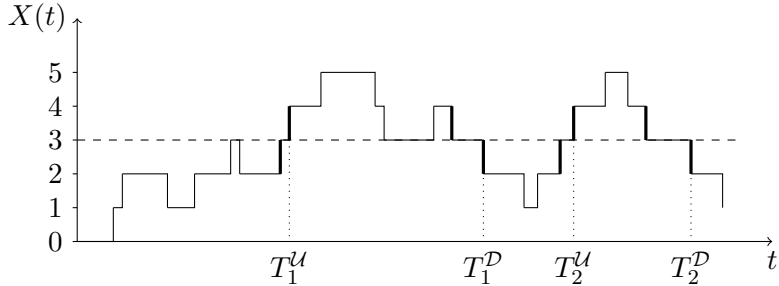
Note that this result does not extend to three (or more) consecutive up (or down) transitions. For example, the rate of immediate transition from $n-1$ to $n$ and then
immediately to $n+1$ and $n+2$, does not coincide with the corresponding rate from $n+2$ to $n-1$.
\end{example}


\section{Preliminaries} \label{notations}

Let  $Y_{\l}$ and $X$ be two independent generic random variables such that $Y_{\l}\sim \text{exp}(\l)$, $X$ is nonnegative and the cumulative distribution function (CDF),
Laplace-Stieltjes transform (LST) and mean of $X$ are denoted by $F$, $F^*$, and $\overline{x}$, respectively. We are interested in the random variable that is defined as the difference $X-Y_{\l}$ given that $X \ge Y_{\l}$. We further assume that $P(X \in (0,\infty))>0$ (or equivalently $\overline{x} > 0$) and hence, the event $X \ge Y_{\l}$ is non-empty. Denote the CDF, LST and mean of this random variable by $D_{\l,F}$, $D^*_{\l,F}$, and $\overline{d}_{\l,F}$, respectively.
Straightforward calculations lead to
\begin{equation}\label{DCDF}
D_{\l,F}(w)=\frac{\int_{u=0}^{w}  \l e^{\l u}\int_{x=u}^\infty{e^{-\l x}\,dF(x)}\,du}{1-F^*(\l)}, \quad w \ge 0,
\end{equation}

\begin{equation}\label{DLST}
D^*_{\l,F}(s)=\frac{\l}{1-F^*(\l)}\frac{F^*(s)-F^*(\l)}{\l-s}, \quad Re(s) \ge 0,
\end{equation}
and
\begin{equation}\label{DMEAN}
\bar{d}_{\l,F}=\frac{\overline{x}}{1-F^*(\l)}-\frac{1}{\l} .
\end{equation}

\begin{theorem}\label{unique}
	Let $H$ be a CDF with $H(0) = 0$ and probability density function (PDF) $h$ such that $0<h(0)<\infty$. Then, for each $\l>0$, there exists a unique CDF $F$ such that $H=D_{\l,F}$.
\end{theorem}

\begin{proof}
	Fix $\l > 0$ and consider a CDF $F$. Let $F^*$ and $H^*$ be the LSTs associated with $F$ and $H$, respectively. From \eqref{DLST} and \eqref{DCDF} we learn that
	\begin{equation}\label{con}
	\frac{\l}{1-F^*(\l)}\frac{F^*(s)-F^*(\l)}{\l-s}=H^*(s), \quad Re(s) \ge 0,
	\end{equation}
	and
	\begin{equation}\label{con1}
	\frac{\int_{u=0}^{w}  \l e^{\l u}\int_{x=u}^\infty{e^{-\l x}\,dF(x)}\,du}{1-F^*(\l)}=H(w), \quad w \ge 0,
	\end{equation}
		are two (equivalent) necessary conditions for $H=D_{\l,F}$. Taking the derivative with respect to $w$ of \eqref{con1} and setting $w=0$ gives
	
	\begin{equation}\label{gamma}
	\l \frac{F^*(\l)}{1-F^*(\l)}= h(0) =:\gamma .
	\end{equation}
	Solving for $F^*(\l)$ gives
	\begin{equation}\label{Fstar}
	F^*(\l)=\frac{\gamma}{\l+\gamma}.
	\end{equation}
	Finally, \eqref{Fstar} and \eqref{con}  uniquely define the LST $F^*$ as
	\begin{equation}\label{Fstars}
	F^*(s)=H^*(s)\frac{\l-s}{\l}\frac{\l}{\l+\gamma}+\frac{\gamma}{\l+\gamma}
	\end{equation}
\end{proof}

\begin{remark}
	Let $H$ be such that $h(0)$ equals zero or infinity. Equation \eqref{gamma} implies that $F^*(\l)$ equals zero or one, respectively, for any $F$ such that $H=D_{\l,F}$. Observe that $F^*(\l)=\int_0^\infty{e^{-\l x}\,dF(x)}$ is the probability that a nonnegative random variable $X$ with CDF $F$  is less than an independent exponential random variable with rate $\l$. Hence, $F^*(\l)$ being equal to zero or one contradicts the assumption $P(X\in(0,\infty))>0$.
\end{remark}


\section{The G/M/1 queueing model} \label{GM1}
In this section we derive some well-known results on the G/M/1 queue using the rate balance principle.
Consider a G/M/1 queue where $G$, $G^*$, and $1/\l$ are the CDF, LST, and mean of the inter-arrival times, respectively. Service times are exponential with rate $\m$.  Assume that the system is stable, so $\l < \m$, and denote the steady state probability of having $n$ customers in the system just before an arrival, or just after a departure
instance by $a_n$. Denote by $\pi_n$, $n\ge 0$, the steady state probability of having $n$ customers in the system at an arbitrary instance, and
let $R_n$ and $R^*_n$ denote the CDF and LST, respectively, of the steady state distribution of the residual inter-arrival time at a departure instance, conditioned on the queue length $n$.

\begin{theorem} \label{GM1Rn}
	For $n\ge 0$, $R^*_n$ is not a function of $n$. In other words, the residual inter arrival time and the number of customers in the queue at departure instances are independent.
\end{theorem}

\begin{proof}
Since service times are memoryless, the queue length distribution is insensitive to the service regime, as long as it is work conserving and non-anticipating. Hence we
	assume w.l.g. that the service regime is Last Come First Served with preemption (LCFS-PR)
and consider such a system at departure instances. The number of customers just after departure equals the number of customers that the departing customer saw upon her arrival. Observe that this number is independent of any random variable solely defined by the period that begins at her arrival and ends at her departure, such as, e.g., the residual inter-arrival time at her departure.
\end{proof}
\\

In the remainder of this section we denote $R_n$ and $R^*_n$, $n\ge 0$, by $R$ and $R^*$, respectively. For its explicit formula we refer to Remark~4.1 at the end of this section.
\cite{HK11,sindo} consider the similar residual, but then at arbitrary instances, and show that the independence in Theorem~\ref{GM1Rn}
extends only to where it is given that $n \geq 1$. Hence, two distributions are derived, one for the case where $n=0$ and one for the case
where $n \geq 1$. The following theorem presents a well known result for the steady state probabilities $a_n$ and $\pi_n$, though the characterization of $\sigma$ is new.
\\


\begin{theorem} \label{pirecgm1}
	Let $\s =G^*(\m)/(1-R^*(\m))$. Then,
	
	\begin{equation} \label{pirec}
		\frac{\pi_{n+1}}{\pi_n}=\frac{a_{n}}{a_{n-1}}=\s,~n\ge 1.
	\end{equation}
In particular, the above two ratios are not function of $n$ as long as $n \geq 1$.
\end{theorem}

\begin{proof}
In order to initiate an $n$-two-step up transition, there must be a transition from state $n-1$ to state $n$, i.e., an arrival who finds $n-1$ customers in the system. Of course, such an arrival may be followed by a  departure, and in that case an $n$-two-step up transition will not occur. Otherwise, if this arrival is followed by an additional arrival, this will indeed form an $n$-two-step up transition. Hence, the $n$-two-step up transition rate is equal to the rate of arrivals who find $n-1$ customers in the system, $a_{n-1} \l$, multiplied by the probability that such arrival will be followed by an additional arrival before an exponential service completion, which is $\int_{0}^{\infty}{e^{-\m x}\,dG(x)} =G^*(\m)$.
In a similar way,
the $n$-two-step down transition rate is equal to the rate of departures who leave behind $n$ customers, that (by level crossing) equals to the rate of arrivals who find $n$ customers in the system, $a_n \l$, multiplied by the probability that such departure will be followed by a consecutive departure before the next arrival. Theorem \ref{GM1Rn} implies that the distribution of the residual inter-arrival time at the moment of departure, is equal to $R$ with LST $R^*$. Hence the probability that a consecutive exponential departure takes place before the next arrival equals $\int_{0}^{\infty}{(1-e^{-\m x})\,dR(x)}=1-R^*(\m)$.
Summarizing, by RBP,

\begin{equation}
a_{n-1}\l G^*(\m)=a_{n}\l (1-R^*(\m)),~~~n\ge 1,
\end{equation}
which leads to

\begin{equation}
a_{n}/a_{n-1} = G^*(\m)/(1-R^*(\m)),~~~n\ge 1.
\end{equation}
This completes proof for the first part of theorem. The second part immediately follows by observing that
the rate of arrivals who find $n-1$ customers in the system, $a_{n-1} \l$, equals (by level crossing) to the rate of departures who leave behind $n-1$ customers, $\pi_n \m$.
\end{proof}

The above theorem implies the following corollary (see, e.g., p. 100 in \cite{Hav13}).
\begin{corollary} \label{pidiscoro}
	The fact that $\pi_0=1-\r$ and \eqref{pirec} imply that
	\begin{equation} \label{arrdist}
	a_n=(1-\s)\s^n, ~~~~n\ge 0
	\end{equation}
	\begin{equation}
	\pi_n=\r (1-\s)\s^{n-1},~~~~ n\ge 1.
	\end{equation}
	Moreover, the geometric distribution in \eqref{arrdist} implies that under the FCFS service regime, the sojourn time is the sum of a geometric (random) number
 with parameter $1-\sigma$ of i.i.d exponential variables with parameter $\m$, the distribution of which is exponential with rate $\m(1-\s)$.
\end{corollary}

The theorem below is the ``usual'' characterization of $\sigma$ (see, e.g., p. 100 in \cite{Hav13}), for which
we now suggest a new proof.
\begin{theorem}
	$\s$ is the unique value obeying  	
	\begin{equation} \label{sigma}
	\s=G^*(\m(1-\s))
	\end{equation}
	and $0<\s<1$.
\end{theorem}
\begin{proof}
	We show that the two sides of \eqref{sigma} are probabilities of the same event. The left hand side is the probability of the event that an arbitrary arrival finds a non-empty queue, i.e., $1-a_0=\s$. Now, assume w.l.g that the service regime is FCFS. In that case, the sojourn time is exponential with rate $\m(1-\s)$ (see Corollary \ref{pidiscoro}). The event that an arbitrary arrival finds a non-empty queue is also equal to the event that the sojourn time of the previous arrival is greater than her inter-arrival time. The probability of this event equals $\int_0^\infty{e^{-\m(1-\s)x}\,dG(x)}=G^*(\m(1-\s))$, which is the right hand side of \eqref{sigma}. The uniqueness of the solution of (\ref{sigma}) can be argued using standard convexity arguments (see e.g. p. 101 in \cite{Hav13}).
\end{proof}

\begin{theorem} \label{resmg1dist}
	The steady-state CDF of the residual inter-arrival time at a departure instance equals $D_{\m(1-\s),G}$.
\end{theorem}
\begin{proof}
	Assume w.l.g the FCFS service regime. Under this regime, the sojourn time is exponential with rate $\m(1-\s)$ (see Corollary \ref{pidiscoro}). Tag a customer and let $S$ to be her sojourn time, $A$ be the next inter-arrival time after her arrival, and $R$ be the residual inter-arrival time at her departure. Of course, $A$ and $S$ are independent. If $S<A$, then $R=A-S$. Otherwise, if $S>A$, the memoryless property of $S$ implies that the remaining sojourn time at the moment of the next arrival, $S-A$, is again distributed as $S$ and that implies that the distribution of the residual at the moment of departure is the same as the distribution of $R$. The above and the law of total probability imply that,
	$$
	P(R<x)=P(A-S<x|S<A)P(S<A)+P(R<x)P(S>A).
	$$
	Solving for $P(R<x)$ gives
	$$
	P(R<x)=P(A-S<x|S<A)
	$$
	as required.
\end{proof}
\begin{remark}
	Theorem~\ref{resmg1dist} implies that $R^*(s)=D^*_{\m(1-\s),G}(s)$. Equations \eqref{DLST} and \eqref{sigma} implies that,
	\begin{equation} \label{Rstar}
	R^*(s)=\frac{\m(G^*(s)-\s)}{\m(1-\s)-s}.
	\end{equation}
	In particular, $R^*(\m)=1-\frac{G^*(\m)}{\s}$, which coincides with the definition of $\s$ in Theorem \ref{pirecgm1}.
\end{remark}
\begin{remark}
	Observe that $R_0$ is the idle period in this model. Hence \eqref{Rstar} gives the LTS of the idle period, which coincides with Theorem 1 in \cite{ABP2005}.
\end{remark}

\section{The G/Mn/1 queueing model} \label{GMn1}
Here we deal with a variation of the G/M/1 model where the service rate is queue length dependent. More formally, given that there are $n$ customers in the system at time $t$, the number of service completions within the time interval $[t,t+\Delta]$ is independent of the past. Moreover, the probability of a single service completion within this time interval equals $\m_n\Delta +o(\Delta)$ and the probability of two or more service completions equals $o(\Delta)$. The arrival process is as in the G/M/1 model. Except for the service rates, we use the same notations as done in the previous section.

In this section we use the RBP in order to show the following. First, we show that Theorem \ref{pirecgm1}
can be generalized to the G/Mn/1 model, where $\m_n$ replaces $\m$. Then, we show that the probability that a customer who leaves behind $n$ customers upon departure, is the first to depart during the current inter-arrival period equals $1-G^*(\m_{n+1})$ (which is the probability that an inter-arrival period exceeds an exponential random variable with parameter $\m_{n+1}$). Finally, we derive an original recursion for the CDF of $R_n$, $n \geq 1$. In general, this recursion does not have easily computable initial conditions. Yet, we show that this can be overcome when $\m_n$ becomes constant once $n$ is large enough.

\begin{theorem} \label{Gmn1pi}
		\begin{equation} \label{Gmn1pieq}
		\frac{\pi_{n+1}\m_{n+1}}{\pi_n\m_{n}}=\frac{a_{n}}{a_{n-1}}=\frac {G^*(\m_n)}{1-R_n^*(\m_n)},~n\ge 1.
		\end{equation}
\end{theorem}

\begin{proof}
	The same arguments as in the proof of Theorem \ref{pirecgm1} yield that the rates of $n$-two-step up and down transitions, $n\ge 1$, equal $a_{n-1}\l G^*(\m_n)$ and $a_{n}\l (1-R^*(\m_n))$, respectively. Equating these rates leads to the second equality in (\ref{Gmn1pieq}).
	Similarly, by level crossing we get that $a_{n-1}\l = \pi_n \m_n$, yielding the first equality in (\ref{Gmn1pieq}).
\end{proof}

\begin{corollary}\label{corGmn1pi}
For $n \ge 1$,
\[
a_n = a_0 \prod_{k=1}^n \frac {G^*(\m_k)}{1-R_k^*(\m_k)}, \quad
\pi_{n+1} = \pi_1 \frac{\mu_1}{\mu_{n+1}} \prod_{k=1}^n  \frac {G^*(\m_k)}{1-R_k^*(\m_k)}.
\]
\end{corollary}

The next theorems give a recursion for $R^*_n$, $n\ge 0$, that is required in order to use the result of Corollary \ref{corGmn1pi}. The analysis here is a dual version of the analysis of the Mn/Gn/1 model in the next section.

\begin{lemma}\label{firstGMn1}
	The probability that a departure who leaves behind $n\ge 0$ customers in the system, is the first to depart during the current inter-arrival time equals $1-G^*(\m_{n+1})$. In particular, this probability equals the probability that an inter-arrival time who commences when $n+1$ customers are in the system will end after completion of the customer currently in service.
\end{lemma}


\begin{proof}
We use the same TST argument as in the proof of Theorem \ref{Gmn1pi} (and Theorem \ref{pirecgm1}). We already argued that the rate of $n\!+\!1$-two-step up transitions equals $a_{n}\l G^*(\mu_{n+1}) = \pi_{n+1} \mu_{n+1} G^*(\mu_{n+1})$, but the rate of $n\!+\!1$-two-step down transitions can now be argued differently than before. An $n\!+\!1$-two-step down transition occurs when and only when a departure leaves behind $n$ customers in the system, and she is not the first to depart during the current inter-arrival time. Hence the rate of $n\!+\!1$-two-step down transitions equals the rate of departures leaving behind $n$ customers, $\pi_{n+1} \mu_{n+1}$, times the probability that such departure is not the first to depart during the current inter-arrival time, i.e., one minus the probability of being the first. Now, the result follows by equating the two-step up and down transition rates.
\end{proof}
\\

\begin{remark}\label{altproof} {\bf Alternative proof of Lemma \ref{firstGMn1}}\\
For each departure who leaves behind $n$ customers in the system, there is an arrival who finds $n$ customers in the system, or, in different words, an inter-arrival time that was initiated with $n+1$ customers. Hence, the rate of such departures, denoted by $\delta_n$, and the rate of such arrivals, denoted by $\a_n$, are equal. Furthermore, for each departure who leaves behind $n$ customers in the system and is the first to depart during the current inter-arrival time, there is an inter-arrival time that was initiated with $n+1$ customers, and had at least one departure during it. Hence, the rate of such first departures, denoted by $\delta_n^f$, and the rate of such inter-arrival times, denoted by $\a_n^+$, are equal. The probability we are after equals the proportion of the rate of departures who leave behind $n$ and are first to depart out of the total rate of departures who leave behind $n$ customers. In the above notation it equals to ${\delta_n^f}/{\delta_n}$, and the explanation above implies that it equals to  ${\a_n^+}/{\a_n}$. The last term equals the probability that an inter-arrival time that was initiated with $n+1$ customers in the system, will have at least one departure during it. This probability equals the probability that a generic random variable with CDF $G$ is greater than an exponential random variable with rate $\m_{n+1}$, which is equal to $\int_{0}^{\infty}{(1-e^{-\m_{n+1} x)}\,dG(x)}= 1-G^*(\m_{n+1})$.
\end{remark}

\begin{theorem} \label{Rstarrec}
	\begin{equation} \label{Rrecgmn1}
	R^*_n(s) = (1 - G^*(\m_{n+1}))D^*_{\m_{n+1},G}(s)+G^*(\m_{n+1})D^*_{\m_{n+1},R_{n+1}}(s),~~n\ge 0
	\end{equation}
\end{theorem}
\begin{proof}
	Consider a departure who leaves behind $n$ customers in the system and is the first to depart during the current inter-arrival time. The fact that this customer is the first to depart implies that the distribution of the residual inter-arrival is $D_{\m_{n+1},G}$. Now, consider an arrival who leaves behind $n$ customers in the system and is not the first to depart during the current inter-arrival time. That means that there is a customer that departed before her, and left behind $n+1$ customers. From the moment of the previous departure, a residual inter-arrival with distribution of $R_{n+1}$ was initiated, and the current departure is the first to depart during it. Hence, in this case, the distribution of the residual inter-arrival time is $D_{\m_{n+1},R_{n+1}}$. Finally, using the law of total probability along with the probability in Lemma~\ref{firstGMn1} completes the proof.
\end{proof}

\begin{remark}
The recursion in \eqref{Rrecgmn1} can be used in both directions. Of course, if $R^*_{n+1}$ is in hand, one can apply the recursion to get $R^*_n$. Likewise, if $R^*_n$ is in hand, one can solve \eqref{Rrecgmn1} for $D^*_{\m_n,R_{n+1}}$ and use \eqref{Fstars} to get $R^*_{n+1}$.
\end{remark}

The above theorem gives a recursion, but does not specify a starting point, i.e., $R_k$ for some $k\ge 0$ that can be used in order to apply \eqref{Rrecgmn1} recursively.  Unfortunately, for the general case, such starting point is not available, but the following theorem implies that the complexity of the computation of each $R_n$, $n\ge 0$, does not depend on $n$.

\begin{theorem} \label{GMn1resind}
Let $R_n$, $n \ge 0$ be the CDF of the conditional residual inter-arrival time at departure instance in a G/Mn/1 queueing model with service rates $\mu_n$. For $k\ge 0$, let $R^{(k)}_n$, $n\ge 0$, be the CDF of the residual inter-arrival time at departure instance, conditioned on queue length $n$, associated with a G/Mn/1 model with service rates $\m^{(k)}_n$, such that $\m^{(k)}_n=\m_{n+k}$, $n\ge 1$. Then, \begin{equation}
R_m^{(k-m)}=R^{(k)}_{0},~~~0\le m \le k.
\end{equation}
In particular,
$$
R_k=R^{(k)}_{0},~~~k\ge 0.
$$
\end{theorem}
\begin{proof}
	We use a similar argument as in the proof of Theorem \ref{GM1Rn}. Consider the system associated with service rates $\mu_n^{(k-m)}$, for some $0\le m \le k$. Assume w.l.g that the LCFS-PR service regime is used. The residual inter-arrival time at the departure of a customer who leaves behind $m$ customers is a function of the service process only from her arrival to her departure. This process is stochastically equivalent to the process of a customer who finds $0$ customers in a system with service rates $\m^{(k)}_n$.
\end{proof}

 \begin{corollary} \label{shift}
For any fixed $k\ge 0$,
$$
R_{k+m}=R^{(k)}_m,~~~m\ge 0.
$$
Moreover, \eqref{Gmn1pieq} implies that the steady state distribution of the difference between the number of customers in the system and $k$, conditioned on the queue length being greater than or equal to $k$, equals the steady state queue length distribution in the G/Mn/1 model with service rates $\mu_n^{(k)}$.
 \end{corollary}

\begin{remark}
Consider the special case, where from some (arbitrary large) queue length, service rates are equal. Specifically,  assume that there exists an $N\ge 1$ and $\bar{\m}$ such that $\m_n=\bar{\m}$, for all $n\ge N$.
	An immediate consequence of Corollary \ref{shift} is that for $n\ge N$, $R_n=\bar{R}$, where $\bar{R}$ is the steady state residual inter-arrival time distribution in a G/M/1 model with constant service rate $\bar{\m}	$. This means that,
starting with $R^*_N$ from \eqref{Rstar}, the transforms $R^*_{N-1},R^*_{N-2},\dots,R^*_{0}$ can be computed recursively using \eqref{Rrecgmn1}.
\end{remark}

\begin{example} {\bf G/M/c model}\\
	Consider the G/M/c model. This model is probabilistically equivalent to the G/Mn/1 model with $\m_n=n\m$, $1\le n \le c$, and $\m_n=c\m$, $n>c$. Using the result in Theorem \ref{GMn1resind} and \eqref{Rstar} we get that $$R^*_c(s)=\frac{c\m(G^*(s)-\s)}{c\m(1-\s)-s},$$ where  $\s$ is the unique solution of $\s=G^*(c\m(1-\s))$ and $0<\s<1$. Using the results from Corollaries~\ref{shift} and~\ref{pidiscoro} we get that
$$\pi_n=C\s^n, ~~n\ge c$$
for some $C>0$. The remaining probabilities, $\pi_n$, $0 \le n <c$, can be computed using \eqref{Gmn1pieq} and \eqref{Rrecgmn1}. Observe that this method does not require to compute infinite sums, as in the embedded Markov chain analysis of this model, see e.g. p. 348 in \cite{asmus}.
\end{example}
\section{The Mn/Gn/1 queueing model} \label{MG1}

The model dealt with in this section is similar to the model analyzed in \cite{K2008}, with the addition that service time distribution is state dependent.
Consider a single server queueing system where service times are independent but not necessarily identically distributed. The distribution of a service time depends on the state of the system upon service commencement. Specifically, the  distribution of a service time that commence with $n\ge 1$ customers in the system is with CDF $G_n$ and LST $G^*_n$. The arrival process is as follows. Given that there are $n$ customers in the system at time $t$, the number of arrivals within the time interval $[t,t+\Delta]$ is independent of the past. The probability of a single arrival within this time interval equals $\l_n\Delta +o(\Delta)$ and the probability of two or more arrivals equals
$o(\Delta)$.

Assume that the system is stable. Let $\pi_n$, $n\ge 0$, be the steady state probabilities of having $n$ customers in the system at arbitrary instance, and let $R_n$ and $R^*_n$, $n\ge 1$ be the CDF and LST, respectively, of the residual service time at an arrival instance, given that there are $n$ customers in the system.


We start this section by a recursion for $\pi_n$, which does not have easily computable initial conditions. We then show that for $n \geq 2$, an arrival, who sees $n$ customers upon his arrival, is with probability $1-G^*(\m_n)$ the first to appear during the current service period. We end with a recursion on $R_n$, $n \geq 1$, and provide simple initial conditions.

The above model is similar to the Mn/Gn/1 model in \cite{BAR15}, except that service rates in \cite{BAR15} are not constant, but can change at arrival instances. In \cite{BAR15} the steady state probabilities and a recursion for the residual conditional service time are derived using the method of supplementary variables~\cite{cox}. We derive these results for our related model (though the extension to the model in \cite{BAR15} is straightforward) by using probabilistic and RBP arguments, similar to the ones in the previous sections.
We further like to point out that \cite{econo} also uses probabilistic arguments (but different from the ones used below) to derive the results in \cite{K2008} for the Mn/G/1 model.

\begin{theorem} \label{pirecGM1}
	The steady state probabilities for the Mn/Gn/1 model obey the following birth and death like equations:
	\begin{equation}\label{MnG1pi}
	\pi_{n-1}\l_{n-1}(1-R_{n-1}^*(\l_n))=\pi_n\l_n G_n^*(\l_n),~~~n\ge 1
	\end{equation}
	where $R_0^*=G_1^*$.
\end{theorem}
\begin{proof}
	We use TST as in the analysis of the G/M/1 model. In order to initiate an $n$-two-step up transition, there must be a transition from state $n-1$ to state $n$, i.e., an arrival who finds $n-1$ customers in the system. This arrival should be followed by an additional arrival in order for an $n$-two-step arrival to occur. Hence, the $n$-two-step up transition rate is equal to the rate of arrivals who find $n-1$ customers in the system, $\pi_{n-1} \l_{n-1}$, multiplied by the probability that such arrival will be followed by an additional arrival. At the moment of the first arrival, for $n>1$, a residual service time is initiated, and for $n=1$ a fresh service is initiated. The probability of the event that an additional arrival occurs before this residual (or fresh) service time equals the probability that a generic random variable with distribution $R_{n-1}$ (or $R_0=G_1$, respectively) is greater than an exponential random variable with rate $\l_n$, given by $1-R_{n-1}^*(\l_n)$. Similarly, in order to initiate an $n$-two-step down transition, there must be a transition from state $n+1$ to state $n$, i.e., a departure who left behind $n$ customers in the system. Of course, such a departure must be followed by a consecutive departure in order to form an $n$-two-step down transition. Hence, the $n$-two-step down transition rate is equal to the rate of departures who leave behind $n$ customers, that equals to the rate of arrivals who find $n$ customers in the system, $\pi_n \l_n$, multiplied by the probability that such departure will be followed by a consecutive departure before the next arrival. At the moment of the first departure, a fresh service time is initiated. Hence, the probability of a second departure equals the probability that a generic random variable with distribution $G_n$ is less than an exponential random variable with rate $\l_n$. This probability equals $G_n^*(\l_n)$.
\end{proof}

\begin{corollary}\label{corpirecGM1}
For $n \ge 1$,
\[
\pi_n = \pi_0 \frac{\l_{0}}{\l_n} \prod_{k=1}^n  \frac {1-R_{k-1}^*(\l_k)}{ G_k^*(\l_k)}.
\]
\end{corollary}

In order to use Corollary \ref{corpirecGM1} for calculation of the steady-state probabilities, the LSTs of the conditional residual service times are required. These LSTs can be derived using the  method  in \cite{K2008} for the Mn/G/1 model, but we  propose an alternative probabilistic approach.

\begin{lemma}\label{first}
	The probability that an arrival who finds $n\ge 2$ customers in the system, is the first to arrive during the current service equals $1-G_n^*(\l_n)$. In particular, this probability equals the probability that a service who commences when $n$ customers are in the system will be completed after the next arrival.
\end{lemma}
\begin{proof}
We use the same TST argument as in the proof of Theorem \ref{pirecGM1}. We already argued that the rate of $n$-two-step down transitions equals $\pi_n \l_n G^*_n(\l_n)$, but the rate of $n$-two-step up transitions can now be argued differently than before. An $n$-two-step up transition occurs when and only when an arrival finds $n$ customers in the system, and she is not the first to arrive during the current service. Hence, the rate of $n$-two-step up transitions equals the rate of arrivals who find $n$ customers, $\pi_n \l_n$, times the probability that such arrival is not the first to arrive during the current service time, i.e., one minus the probability of being the first. Now, the result follows by equating the up and down transition rates.
\end{proof}
\\

\begin{remark} {\bf Alternative proof of Lemma \ref{first}}\\
Just as in Remark \ref{altproof}, an alternative proof can be given by calculating
the proportion of the rate of arrivals who find $n$ and are first to arrive out of the total rate of arrivals who find $n$ customers upon arrival.
\end{remark}

\begin{theorem}
	The conditional residual service time distributions $R_n$, $n \geq 1$, follow the recursion
		\begin{equation} \label{inimng}
		R_1 = D_{\l_1,G_1}
		\end{equation}
		and
		\begin{equation} \label{recmng}
		R_n = (1-G_n^*(\l_n)) D_{\l_n,G_n}+ G_n^*(\l_n)D_{\l_n,R_{n-1}}, \quad n \geq 2.
		\end{equation}

\end{theorem}

\begin{proof}
	Consider an arrival who finds $n$ customers in the system and is the first to arrive during the current service. The fact that this customer is the first to arrive implies that the distribution of the residual service time is equal to $D_{\l_n,G_n}$. Now, Consider an arrival who sees $n$ customers in the system and is not the first to arrive during the current service. That means that there is a customer that arrived before her, and saw $n-1$ customers. From the moment of the previous arrival, the residual service time is distributed as $R_{n-1}$, and the current arrival is the first to arrive during this residual service time. Hence, the distribution of the residual service time at the current arrival is $D_{\l_n,R_{n-1}}$. Finally, use of the law of total probability along with the probability in Lemma~\ref{first} completes the proof.
\end{proof}

\section{Summary} \label{summ}

In this paper we introduced a rate balance principle for general stochastic processes. For the special case of birth and death like processes, we showed that it leads to a new balance principle between the rates of two consecutive up transitions and of two consecutive down transitions, and demonstrated the potential of this ``two step transition'' rate principle to probabilistically derive some well known, but also new results for the G/Mn/1 and Mn/Gn/1 queues, such as recursions on the conditional distributions of the residual service and inter-arrival times, given the queue length. Hence, we believe that this principle is a promising tool to also explore other stochastic processes.

\begin{center}
	{\bf Acknowledgement}
\end{center}
This research was partly supported by Israel Science Foundation grant no. 1319/11.

\end{document}